\documentclass[10pt]{article}

\usepackage[english]{babel}
\usepackage[utf8]{inputenc} 
\usepackage{amsfonts}
\usepackage[square,sort,comma,numbers]{natbib}
\bibpunct{[}{]}{;}{a}{,}{,}
\usepackage{hyperref}

\usepackage{mathtools}
\usepackage{amsthm}
\usepackage{amssymb}
\usepackage{amsmath}
\usepackage{needspace}
\usepackage{etoolbox}
\usepackage{lipsum}
\usepackage{authblk}
\usepackage{color,soul}
\usepackage{blindtext}
\usepackage[svgnames]{xcolor} 
\usepackage{xcolor}

\newcounter{theo}
\newtheorem{lmm}[theo]{Lemma}
\newtheorem{thr}{Theorem}

\newtheorem{crl}[theo]{Corollary}
\newtheorem{prp}[theo]{Proposition}

\newtheorem{defn}[theo]{Definition}
\newtheorem{rmrk}[theo]{Remark}

\DeclareMathOperator*{\argmax}{arg\,max}

\newcommand{\R}{\mathbb{R}}

\newcommand{\N}{\mathbb{N}}
\newcommand{\Z}{\mathbb{Z}}

\newcommand{\Prob}{\mathbb{P}}

\newcommand{\pr}{\eta}
\newcommand{\bd}{\partial}

\newcommand{\hm}{\omega}

\newcommand{\cp}{\mathrm{Cap}}
\newcommand{\rad}{\mathrm{R}}

\newcommand{\eqv}{\asymp}
\newcommand{\E}{\mathbb{E}}

\title{How long are the arms in DBM?}

\author{Ilya Losev 
\thanks{losevilya14@gmail.com; University of Cambridge, United Kingdom},\,
Stanislav Smirnov 
\thanks{sksmirnov@gmail.com; Universit\'e de Gen\`eve, Geneva, Switzerland}
\thanks{St. Petersburg University, Russia}
\thanks{Skolkovo Institute of Science and Technology, Russia}}

\begin{document}

\maketitle

\begin{abstract}
Diffusion Limited Aggregation and its generalization, Dielectric Breakdown model play an important role in physics, approximating a range of natural phenomena. 
Yet little is known about them, with the famous Kesten's estimate on the DLAs growth being perhaps the most important result. 
Using a different approach we prove a generalisation of this result for the DBM in $\Z^2$ and $\Z^3$. 
The obtained estimate depends on the DBM parameter, and matches with the best known results for DLA.
In particular, since our methods are different from Kesten's, our argument provides a new proof for Kesten's result both in $\Z^2$ and $\Z^3$.
\end{abstract}

\section{Introduction}

Diffusion limited aggregation (DLA) \citep{WitSan} was introduced 
as a model of mineral deposition and electrodeposition and has been a great challenge for mathematicians ever since. This model is believed to exhibit non-equilibrium fractal growth, producing highly irregular, branching fractal clusters. Although this observation comes from numerous simulations (see, e.g. \citep{GrebBel, Meakin}), there are very few rigorous theoretical results explaining these phenomena. This is in sharp contrast to internal DLA, which is known to converge to a disk in shape \citep{LBG}.

Another stochastic growth process, called dielectric-breakdown model (DBM) \citep{NPW}, was introduced as a model of such physical processes as lightnings, surface discharges, and treeing in polymers. It can be viewed as a one-parameter generalisation of DLA. It is conjectured that DBM also produces irregular clusters in a similar manner (see, e.g. \citep{NPW, MMJB}). 

It is believed that some other models, including Hastings-Levitov model \citep{HastingsLevitov}, are in the same universality class as DBM (for recent progress on Hastings-Levitov model see \citep{NST1, NST2, Norris_2012, Sola_2019, Rohde, STV18, STV10}). 
Hastings-Levitov model can also be defined on the upper half-plane, see \citep{Stationary_HL} for rigorous results in this setting.

\medskip

The main questions in the area include finding scaling limits of these models and describing 
their fractal properties. Essentially, the only rigorous result in this direction is due to Kesten \citep{Kesten}, who showed
that DLA clusters are not growing too fast (see \citep{Kesten_higher_dim} and \citep[Chapter 2.6]{LawlerInersections} for generalisations to higher dimensions). 
This result can be informally restated to say that DLA in the plane has fractal dimension at least $3/2$. 

Using an alternative approach we generalize Kesten's estimate to the DBM family with parameter $\pr$, showing that its dimension is at least $2-\pr/2$ in the plane. 
This requires new ideas, as original Kesten's DLA proof does not generalize automatically to DBM. 
Specialising our argument to $\pr=1$ we obtain a new proof of Kesten's estimate in 2 and 3 dimensions, which matches with the best known results for DLA, which are due to Kesten in $\Z^2$ \citep{Kesten} and Lawler in $\Z^3$ \citep[Chapter 2.6]{LawlerInersections} (which is an improvement over the original result \citep{Kesten_higher_dim} in 3~dimensions).
Our approach utilizes the dynamical nature of the process and is based on the connection between growth rate and harmonic measure multifractal spectrum observed in \citep{TauGrowth}.

\medskip
Throughout the paper we will be concerned only with DLA and DBM on $\Z^2$ and $\Z^3$, though methods can be modified to obtain results in the continuous setting of unit particles in $\R^2$ or $\R^3$ if we superimpose a square grid of mesh size smaller than the size of particles.

\subsection{Definitions}
Informally speaking, DLA cluster starts as one point and grows as follows. We take a small particle near infinity and let it perform a random walk until it hits the outer boundary of the existing cluster for the first time, attaching it to the cluster at that point. Then we take a new particle and repeat this process all over again for the new cluster. In other words, on each step a new particle is attached at a point with probability equal to harmonic measure in the complement of the existing cluster as viewed from infinity. Similarly, clusters in DBM with parameter $\pr\geq 0$ grow by randomly attaching particles with probability proportional to harmonic measure raised to the power of $\pr$. Observe that DBM with parameter $\pr=1$ coincides with DLA.

\medskip

Now we pass to the formal definitions of these models.
We write $x\sim y$ if $x$ and $y$ are two adjacent vertices. For $A\subset \Z^{d}$ let 
\begin{equation*}
\bd A = \left\{y \in \Z^{d}: y \notin A, \exists x\in A: x\sim y \right\}, \qquad \overline{A} = A \cup \bd A
\end{equation*}
be the set of neighbours of $A$ and the closure of $A$ respectively.
For a finite $X\subset \Z^d$ and $y\in X$ we write $\hm_X(y)=\hm(y, X)$ for the harmonic measure of lattice cite $y$ in the complement of $X$ as viewed from $\infty$. In other words, $\hm(y, X)$ is the probability that a random walk started infinitely far away, conditioned on hitting $X$, hits $X$ for the first time at the point $y$. 

\begin{defn}
Dielectric-breakdown model with parameter $\pr$ (DBM-$\pr$) on $\Z^d$ is a Markov chain $\left\{A_n\right\}_{n\geq 0}$ of finite connected subsets of $\Z^d$ such that 
\begin{equation*}
\Prob\left[A_{n+1} = C|A_n\right] = \frac{\hm(y, \overline{A}_n)^{\pr}}{\sum_{z\in\bd A_n}\hm(z, \overline{A}_n)^{\pr}}, \qquad \text{if } C = A_n \cup \{y\}, \, y\in \bd A_n,
\end{equation*}
and $A_0 = \{0\}$.
Diffusion limited aggregation (DLA) is defined as $DBM$ with parameter $\pr=1$.
\end{defn}
Note that $\hm(y, \overline{A})$ is supported on $\bd A$.
Also notice that DBM with $\pr=0$ corresponds to the Eden model \citep{Eden60}, when the next particle is attached equally likely at any of the boundary sites (here we use the convention $0^0=1$, so that the cluster can grow at sites on the boundary which are not connected to $\infty$).

\medskip
\pagebreak[3]

The outstanding questions about these models are to describe their scaling limits and find their dimensions or, alternatively, growth rates. 
The growth rate $\beta(\pr)$ of DBM-$\pr$ is defined as
\begin{equation}\label{eqGrowthRateDef1}
\beta(\pr) = \lim_{n\to \infty} \frac{\log \rad(A_n)}{\log n},
\end{equation}
where  $\left\{A_n\right\}_{n \geq 0}$ is a DBM-$\pr$ growing cluster, 
and $\rad(A) = \max_{x\in A}|x|$ is the radius of $A$ with respect to the usual Euclidian distance $|\cdot|$.

One can informally argue that dimension $D(\pr)$ is related to the growth rate by
\begin{equation}
D(\pr) = \frac{1}{\beta(\pr)},
\end{equation}
or, in other words, that $D(\pr)$ is such a number that
\begin{equation}\label{eqHausdDimDef}
 \rad(A_n)^{D(\pr)} \eqv n.
\end{equation} 
To make this argument formal one has to show that dimension is well-defined over different scales.

\medskip

It is conjectured that for every $\pr>0$, the limit in \eqref{eqGrowthRateDef1} exists almost surely, is nontrivial (i.e. $\frac{1}{d}<\beta(\pr)<1$ and $1<D(\pr)<d$), and does not depend on the realization of the DBM-$\pr$ cluster, see \citep{MMJB} for numerical simulations. 

\medskip

The main fact known about DLA is due to Kesten, who has proved that DLA cluster radius cannot grow too fast. 
It is immediate that 
$\frac{1}{2} n^{1/d}<\rad(A_n) < n$. The upper bound was improved in \citep{Kesten}. 
\begin{thr}[\citep{Kesten}] \label{ThrKesten}
Let $\left\{A_n\right\}_{n \geq 0}$ be DLA on $\Z^d$. Then there exists a constant $C<\infty$, such that with probability 1
\begin{align*}
\limsup_{n \to \infty} n^{-2/3}\rad(A_n) &< C, \qquad \text{if } d=2;\\
\limsup_{n \to \infty} n^{-2/d}\rad(A_n) &< C, \qquad \text{if }d\geq 3.
\end{align*}
\end{thr}
Essentially, this means that $D(1)\geq 3/2$ in $2$ dimensions, and $D(1)\geq d/2$ in $d\geq 3$ dimensions.
Kesten's Theorem was improved for $d=3$ by Lawler:
\begin{thr}[%
\citep{LawlerInersections}] \label{ThrKestenBetter3D}
Let $\left\{A_n\right\}_{n \geq 0}$ be DLA on $\Z^3$. Then there exists a constant $C<\infty$, such that with probability 1
\begin{equation*}
\limsup_{n \to \infty} n^{-1/2}(\log n)^{-1/4}\rad(A_n)< C.
\end{equation*}
\end{thr}

\begin{rmrk}
The proofs of these Theorems in \citep{Kesten} and \citep{LawlerInersections} do not immediately generalize to DBM-$\pr$ with $\pr\neq 1$, since in this case attachment probabilities have a more complicated form than they do for DLA.
\end{rmrk}

\begin{rmrk} Kesten's argument was revisited and generalized for a larger family of graphs in \citep{KestenBetter} and for ballistic aggregation \citep{Ballistic_aggregation}. 
\end{rmrk} 

The main tool in the proofs of these Theorems is a discrete analogue of Beurling's estimate. 
Namely, the harmonic measure of any given site of a connected set $B$ is bounded from above by $C\rad(B)^{-1/2}$ in $\Z^2$ (see \citep{KestenRW}), and by $C\left(\log \rad(B)\right)^{1/2}\rad(B)^{-1}$ in $\Z^3$ (see \citep[Chapter 2.6]{LawlerInersections}) with an absolute constant $C$. 

\medskip
We give an alternative proof of these results, which does not rely on Beurling's estimate. Instead, we exploit the dynamical structure of the process and observation by Halsey that a well-known formula for an increment of capacity fits nicely into the DLA set-up. It was unrigorously observed in \citep{TauGrowth} that if $\tau$ is a tau-spectrum of the harmonic measure, which can be informally defined as
\begin{equation*}
\displaystyle\sum_{x\in \bd A_n} \hm^{\alpha}(x, A_n) \eqv \rad(A_n)^{-\tau(\alpha)},
\end{equation*}
then Hausdorff dimension of DBM-$\pr$ clusters should be equal to $\tau(\pr+2)-\tau(\pr)$. 
In addition, we combine this with discrete Makarov's Theorem \citep{DiscrMak} to obtain analogous theorem for DBM-$\pr$:

\begin{thr} \label{ThrDiscr2D}
Let $0\leq \pr < 2$ and $\left\{A_n\right\}_{n\geq 1}$ be a DBM-$\pr$ on $\Z^2$. 
Then there exists an absolute constant $\alpha>0$, and $C=C(\pr)>0$ such that with probability $1$ there exists $N$ such that for all $n>N$
$$\rad(A_n) < C n^{\frac{2}{4-\pr}}(\log n)^{\alpha\frac{|\pr-1|}{4-\pr}}. $$ 
\end{thr}
\begin{rmrk}
For $\pr=0$ this result is sharp up to a power of $\log n$, since by trivial lower bound $\frac{1}{2} n^{1/2}<\rad(A_n)$ on $\Z^2$.
\end{rmrk}

We also prove analogous theorem in 3 dimensions.
\begin{thr}\label{ThrDiscr3D}
Let $\left\{A_n\right\}_{n\geq 1}$ be a DBM-$\pr$ on $\Z^3$. 
Then 
\begin{enumerate}
\item For $\pr\geq 1$ $\exists C>0$ such that with probability 1 there exists $N$ such that for all $n>N$ we have $$\rad(A_n) < C n^\frac{\pr}{1+\pr}(\log n)^\frac{\pr}{2(\pr+1)}.$$
\item For $\pr < 1$ $\exists C>0$ such that with probability 1 there exists $N$ such that for all $n>N$ we have $$\rad(A_n) < C n^{1/2}(\log n)^{1/4}.$$
\end{enumerate}
\end{thr}

Our results exactly match the best known estimates for DLA on $\Z^2$ and $\Z^3$, proved in \citep{KestenRW} and \citep[Theorem 2.6]{LawlerInersections}, and are stronger than both Kesten's original result for $\Z^3$ \citep{Kesten_higher_dim} and the result for $\Z^3$ in \citep{KestenBetter}.
However, we do not expect this to be sharp, since these estimates do not match the existing numerical simulations \citep{GrebBel, Meakin28}.

\subsection{Organization of the paper}
\label{sect_org}

We explain our argument at an informal level in Section \ref{HeuristicArgument}, and then prove Theorem \ref{ThrDiscr2D} in Section \ref{Section2Dim}, and Theorem \ref{ThrDiscr3D} in Section \ref{Section3Dim}.

\vspace{15pt}

We will be using the following notation:
\begin{itemize}
\item We write $f \gtrsim g$ for functions $f, g$ if there exists an absolute constant $C>0$ depending on the equation such that $Cf \geq g$;
We write $f\eqv g$ if $f \gtrsim g$ and $f \lesssim g$;
\item $\Prob^x$ denotes the probability measure corresponding to the random walk $S$ started at $x$;
\item $T_A = \min\left\{j> 0: S(j) \in A\right\}$ and $\overline{T}_A = \min\left\{j\geq 0: S(j) \in A\right\}$ are the hitting times of a random walk $S$;
\item $\hm_A(y)= \hm(y, A)$ is the harmonic measure of $y$ in the complement of $A$ as viewed from $\infty$;
\item  $G(x, y, A)$ is the Green's function with poles at $x$ and $y$ in the complement of $A$.
This is the unique function satisfying the following four properties

\begin{enumerate}
\item For any $x, y, \in \Z^d$, $G(x, y, A) = G(y, x, A)$.
\item For any $x\in \Z^d$ and $y\in A$ we have $G(x, y, A) = 0$.
\item For $d=2$, for any $x\in \Z^2$, $G(x, y, A)$ is bounded as a function of $y\in \Z^2$. 
For $d=3$, for any $x\in \Z^3$, $\lim_{y\to \infty}G(x, y, A) = 0$.
\item We have
\begin{equation*}
G(x, y, A) - \frac{1}{4}\sum_{y'\sim y} G(x, y', A) 
= \left\{
  \begin{array}{lr}
  	1, & \text{if } x = y \notin A\\	
    0, & \text{otherwise.} 
  \end{array}
\right. 
\end{equation*}
\end{enumerate}

Note that $G(x, y, A)$ can be defined as the expected number of times that random walk started at $x$ visits $y$ before hitting $A$ for the first time \citep[Chapter 4.6]{LawlerRW},
\begin{equation*}
\displaystyle
G(x, y, A) = \sum_{n=1}^{\infty} \Prob^x\left\{S(n) =y;\, n<\overline{T}_A\right\}.
\end{equation*} 
Also denote $G(x, A) = G(x, \infty, A) =  \lim_{y\to \infty}G(x, y, A)$ in the 2 dimensional case, see e.g. \citep[Section~14, Theorem~3]{Spitzer_book}.
We also remark that there are alternative definitions of $G(x, A)$ in 2 dimensions (see \citep[Proposition~6.4.7]{LawlerRW}), but it is a classical result that these definitions are equivalent, see Appendix~\ref{sect_app}.

\item $\rad(A)$ and $\cp(A)$ are radius and electrostatic capacity (as defined in \citep[Chapter 6.5, 6.6]{LawlerRW}) of $A$ respectively. Recall that in the $2$ dimensional case, 
\begin{equation*}
\cp(A) = \sum_{x\in \bd A} \hm_A(x)a(x-z), \text{ for any } z\in A,
\end{equation*}
where
\begin{equation*}
a(x) = \sum_{n=0}^{\infty}\left[\Prob^0\left\{S(n)=0\right\}-\Prob^0\left\{S(n)=x\right\}\right],
\end{equation*}
and in the $3$ dimensional case
\begin{equation*}\
\cp(A) = \sum_{x\in \bd A} \Prob^x\left\{T_A = \infty \right\}.
\end{equation*}
Also recall (see e.g. \citep[Proposition 6.5.4]{LawlerRW}) that in 3~dimensions
\begin{equation*}
\displaystyle
\hm(x, A) = \frac{\Prob^x\left\{T_A = \infty \right\}}{\cp(A)}.
\end{equation*}

\end{itemize}

\section{Proofs}
\subsection{Heuristic argument}
\label{HeuristicArgument}
We start with an informal account of our argument, which consists of three steps.
\medskip

\textbf{Step 1.}
We use a well-known property of harmonic measure that was first invoked in this context in \citep{TauGrowth}. 
In $2$ dimensional case we observe that if the $(n+1)$-st particle is attached at a point with harmonic measure $\hm$, then
\begin{equation}\label{eqCapIncrementDeterministic}
\cp(A_{n+1}) - \cp(A_n) \eqv \hm^2.
\end{equation}
Hence, for DBM-$\pr$ we have 
\begin{align}
\E \left[ \cp(A_{n+1}) - \cp(A_n) \right] 
&\eqv \sum_{x\in \bd A_n}\hm^{2}(x, A_n)\Prob\left[A_{n+1}\backslash A_n = \{x\}\right] \\
&\eqv \frac{\sum_{x\in \bd A_n}\hm^{\pr+2}(x, A_n)}{\sum_{x\in \bd A_n} \hm^{\pr}(x, A_n)}.\label{HalseyExp2D}
\end{align}

Let $\tau$ be a function such that
\begin{equation}
\sum_{x\in \bd A_n} \hm^{\alpha}(x, A_n) \eqv R^{-\tau(\alpha)}, \label{TauDef}
\end{equation}
where $R$ is the radius of cluster $A_n$.
This function $\tau$ is called \textit{multifractal spectrum} or \textit{tau-spectrum} of the cluster. For the sake of simplicity we are assuming that $\tau$ does not depend much on $n$.

It is well known that $\cp(A_n) \eqv \log n$ (see, e.g.  \citep[Lemma 6.6.7]{LawlerRW}). Differentiating it with respect to $n$ and using \eqref{eqHausdDimDef} and \eqref{HalseyExp2D} with dropped expectation sign we obtain 
\begin{equation*}
R^{-D(\pr)} \eqv \frac{1}{n} \eqv \bd \log n \eqv \bd \cp(A_n) \eqv  R^{\tau(\pr)-\tau(\pr+2)},
\end{equation*}
where $\bd$ denotes derivative in $n$.
Therefore, 
\begin{equation}\label{eqDimTauHalsey}
D(\pr)=\tau(\pr+2)-\tau(\pr).
\end{equation}

\medskip
\textbf{Step 2.}
Let $\sigma$ be such that $\max_{x\in \bd A_n} \hm(x, A_n) \eqv R^{-\sigma}$. Kesten's argument \citep{Kesten} (without Beurling's estimate) implies 
\begin{equation}\label{eqDimIneqKesten}
D(\pr) \geq 1-\tau(\pr) + \pr \sigma.
\end{equation} 
We will briefly recall it for completeness. 
Observe that the longest branch grows with the average speed of at most 
$$\max_{x\in \bd A_n} \frac{\hm^{\pr}(x, A_n)}{\sum_{y\in \bd A_n} \hm^{\pr}(y, A_n)} 
\eqv R^{-\pr\sigma+\tau(\pr)}.$$
Thus, $\bd R \leq R^{-\pr\sigma+\tau(\pr)}$. 
Substituting \eqref{eqHausdDimDef}, we obtain
\begin{equation}
\bd n^{1/D(\pr)} \leq n^{(-\pr\sigma+\tau(\pr))/{D(\pr)}},
\end{equation}
which yields \eqref{eqDimIneqKesten}.

\medskip
\textbf{Step 3.}
After applying trivial inequality $\sigma\geq \tau(\pr+2)/(\pr+2)$ to right-hand side of \eqref{eqDimIneqKesten} and combining with \eqref{eqDimTauHalsey} we obtain
\begin{equation}\label{eqDimCombineKestenHalsey}
\tau(\pr+2)-\tau(\pr) = D(\pr) \geq 1-\tau(\pr) + \frac{\pr\tau(\pr+2)}{(\pr+2)}.
\end{equation}
Hence, 
\begin{equation}\label{eqTauLowerBound}
\tau(\pr+2)\geq(\pr+2)/2.
\end{equation}

Now we apply the discrete version of Makarov's Theorem \citep{DiscrMak}, which states that there are $\eqv n$ vertices with harmonic measure $\eqv 1/n$. This implies that
\begin{equation}\label{eqTauUpperBound}
\tau(\pr) \leq \pr -1.
\end{equation} 
Hence, combining \eqref{eqDimTauHalsey}, \eqref{eqTauLowerBound}, and \eqref{eqTauUpperBound}, we obtain
$D(\pr) \geq (4-\pr)/2.$

\medskip
Analogous heuristics for $3$ dimensional DLA can be found in \citep{Lawler_1994}. The only difference is that the left-hand side of \eqref{HalseyExp2D} is replaced by $$\E \left[ \cp^{-1}(A_{n+1}) - \cp^{-1}(A_n) \right]$$ and we do not have Makarov's Theorem.

\subsection{Proof for DBM in 2 dimensions}
\label{Section2Dim}

In order to make our heuristic argument from Section \ref{HeuristicArgument} rigorous, we translate it from the language of multifractal spectrum $\tau(\alpha)$ back to the language of statistical sums $\sum_{x\in \bd A} \hm^{\alpha}(x, A)$.

Moreover, instead of working with the expectation of the capacity growth  \eqref{HalseyExp2D} we look at the contribution of a given branch inside the cluster to the capacity increments over a long period of time. Informally, this allows us to combine Kesten's argument with observation \eqref{HalseyExp2D} without introduction of the power $\sigma$.

\medskip

First, we justify the capacity increment estimate \eqref{eqCapIncrementDeterministic}.

\begin{lmm} \label{LmmCapGrowDiscr2d}
Let $A\subset\Z^2$ be a compact set and $x\in \bd A.$ Set $B = \left\{x\right\}\cup A$. Then
$\cp (\overline{B}) - \cp (\overline{A}) \eqv \hm^2(x, \overline{A})$.
\end{lmm}
\begin{proof} If $\hm(x, \overline{A})=0$ then $x$ is not accessible by random walk in the complement of $\overline{A}$ starting from $\infty$, so $\cp(\overline{A})=\cp(\overline{B})$.

If $\hm(x, \overline{A})>0$ then $\overline{B} \neq \overline{A}$.
Let $\overline{B} \backslash \overline{A} = \{x_1, \ldots x_k\}$, where $1\leq k\leq 3$.
It is known (see, e.g. \citep[Lemma 6.6.6]{LawlerRW}) that 
\begin{equation}\label{eqLmmCapGrow2dCapIncr}
\cp (\overline{B}) = \cp (\overline{A})+\sum_{j=1}^k\hm(x_j, \overline{B})G(x_j, \overline{A}).
\end{equation}
Recall from Appendix \ref{sect_app} that our definition of the Green's function $G(\cdot, \cdot)$ agrees with the definition in \citep[Proposition 6.4.7]{LawlerRW}.

Let $l$ be the index with the maximal harmonic measure:
$$l = \argmax_j \hm(x_j, \overline{B}).$$ 
Then it is easy to see that
\begin{equation}\label{eqLmmCapGrow2dGreenCompare}
\displaystyle
\hm(x_l, \overline{B})\eqv G(x_l, \overline{A})
\gtrsim G(x_j, \overline{A})
\qquad \text{for all } 1\leq j\leq k.
\end{equation}
Indeed, it is obvious that 
$\hm(x_l, \overline{B})\leq G(x_l, \overline{A}),$
and for any $1\leq j\leq k$ by first-entry decomposition,
\begin{align*}
\displaystyle 
G(x_j, \overline{A})
&
\displaystyle 
= \sum_{i=1}^k\hm(x_i, \overline{B})G(x_i, x_j, \overline{A})
\\
&
\displaystyle
\leq
\hm(x_l, \overline{B}) \sum_{i=1}^k G(x_i, x_j, \{x\}),
\end{align*}
so that $G(x_j, \overline{A})\lesssim \hm(x_l, \overline{B})$ since $\sum_{i=1}^k G(x_i, x_j, \{x\})\lesssim 1$.

Therefore, combining \eqref{eqLmmCapGrow2dCapIncr} and \eqref{eqLmmCapGrow2dGreenCompare} we get
\begin{equation*}
\displaystyle
\cp (\overline{B}) - \cp (\overline{A}) 
\eqv \hm^2(x_l, \overline{B}).
\end{equation*}
Note that by \eqref{eqLmmCapGrow2dGreenCompare} we also have
\begin{equation*}
\displaystyle
 \hm(x_l, \overline{B}) \eqv \sum_{j=1}^k G(x_j, \overline{A}) 
\eqv \hm(x, \overline{A}) ,
\end{equation*}
which finishes the proof.
\end{proof}

It is well-known that for all connected $A\subset \Z^2$ we have (see, e.g. \citep[Lemma 6.6.7]{LawlerRW})
\begin{equation} \label{CapRadDiscr2d}
\left|\,\cp(A) - \frac{2}{\pi} \log\rad(A)\,\right|\lesssim 1.
\end{equation}
This allows us to estimate the harmonic measure of the tip of a given branch in terms of cluster radius.

\begin{crl} \label{CrlBeurlIntDiscr2D}
Let $A_k\subset \Z^2$ for $k\in \N$. 
Assume that $A_{k+1}\backslash A_k = \{x_k\}$ where $x_k \in \bd A_k$ and $\hm_k = \hm(x_k, \overline{A_k})$. 
Then for any $R>0$ and any set of indices $\left\{ {k_j}\right\}_{j=1}^m$ satisfying
$\forall j\leq m: R<|\rad(A_{k_j})|<100R$, we have
$$\left(\prod_{j=1}^{m}\hm_{k_j}\right)^{1/m} \leq C_1 m^{-1/2},$$
for some constant $C_1>0$.
\end{crl}

\begin{rmrk}
This is an integral analogue of the discrete Beurling's estimate \citep{KestenRW}, which states that harmonic measure at every point is less than $R^{-1/2}$, where $R$ is the cluster radius. Although Beurling's estimate is stronger, our proof is shorter and it exploits the dynamical nature of DBM and DLA processes. It would be interesting to adopt our argument to give an alternative proof of the original Beurling's estimate.
\end{rmrk}
\begin{proof}
From Lemma \ref{LmmCapGrowDiscr2d} and estimate \eqref{CapRadDiscr2d} we see that
\begin{multline*}
\sum_{j=1}^m \hm_{k_j}^2 \eqv \sum_{j=1}^m \left(\cp(\overline{A}_{k_j+1}) - \cp(\overline{A}_{k_j})\right) \leq \cp(\overline{}A_{k_m+1}) -  \cp(\overline{A}_{k_1}) \\
\lesssim \log (100R)-\log (R) \lesssim 1.
\end{multline*}
The statement follows from the inequality between geometric and arithmetic means

\begin{equation}\label{ArGeomIneq}
\left(\prod_{j=1}^{m}\hm_{k_j}\right)^{2/m} \leq \frac{\sum_{j=1}^m \hm_{k_j}^2}{m}\lesssim \frac{1}{m}.
\end{equation}
\end{proof}

Now we apply Makarov's Theorem in order to justify \eqref{eqTauUpperBound}.
\begin{lmm} \label{MakApplDiscr} 
There exists $\alpha>0$ such that for any $\pr\geq 0$, there exists $ C_2=C_2(\pr)>0$ such that for all connected sets $A\subset \Z^2$ with $\rad (A)$ big enough we have
$$\sum_{x\in \bd A} \hm(x, \overline{A})^{\pr}>C_2 \frac{ \rad^{1-\pr} (A)}{(\log\rad (A))^{\alpha|{1-\pr}|}}
.$$
\end{lmm}
\begin{proof}
We use Theorem 1.5 from \citep{DiscrMak} which states that for any connected $B\subset \Z^2$ we have
\begin{equation} \label{MakDiscr}
\left|\sum_{x\in B} \hm_B(x)\log \hm_B(x) + \log\rad(B)\right|\lesssim \log \log \rad(B).
\end{equation} 
Since $\sum_{x\in \bd A} \hm(x, \overline{A}) = 1$, by Jensen's inequality for the convex function $\exp(\cdot)$ and weights $\hm(x, \overline{A})$ we have
\begin{multline*}
\sum_{x\in \bd A} \hm(x, \overline{A})^{\pr} = \sum_{x\in \bd A}\hm(x, \overline{A}) \exp\left((\pr-1)\log \hm(x, \overline{A})\right)  \\
\geq \exp\left((\pr-1)\sum_{x\in \bd A} \hm(x, \overline{A})\log \hm(x, \overline{A})\right).
\end{multline*}
Combining this with \eqref{MakDiscr} we get that for some $C, C_2>0$,
\begin{multline*}
\exp\left((\pr-1)\sum_{x\in \bd A} \hm(x, \overline{A})\log \hm(x, \overline{A})\right)
\geq \\
\exp\left((1-\pr)\log \rad(\overline{A}) - C|\pr-1|\log \log \rad(\overline{A}) \right)\geq \\
C_2 \frac{ \rad^{1-\pr} (A)}{(\log\rad (A))^{\alpha|{1-\pr}|}}
,
\end{multline*}
where $\alpha=C$ and $C_2$ accounts for changing $\rad(\overline{A})$ in the left-hand side to $\rad(A)$ in the right-hand side. This finishes the proof.
\end{proof}
\begin{rmrk}
It is believed that, in fact
$$\left|\sum_{x\in A} \hm_A(x)\log \hm_A(x) + \log\rad(A)\right|\lesssim 1,$$
but this has not been proved yet. Unfortunately, the method used in \citep{DiscrMak} is not sufficient to obtain such sharp estimates. 
\end{rmrk}

Now we combine Corollary \ref{CrlBeurlIntDiscr2D} and Lemma \ref{MakApplDiscr} to prove Theorem \ref{ThrDiscr2D}.
\begin{proof}[Proof of Theorem \ref{ThrDiscr2D}]

Let $R=\left\lfloor \rad(A_N)\right\rfloor$. 
We will assume that $N$ is sufficiently large, so that $R>10$.
Set 
\begin{equation} \label{localMdefin}
M := \left\lfloor\widetilde{C}\frac{R^{(4-\pr)/2}}{(\log R)^{\alpha |1-\pr|}}\right\rfloor
\end{equation}
for $\alpha$ from Lemma \ref{MakApplDiscr} and a small constant $\widetilde{C}>0$ given by 
\begin{equation}
\widetilde{C}=(10^4 C_0)^{-1},
\end{equation}
where $C_0=C_1^{\pr} C_2^{-1}\, 2^{\alpha|1-\pr|}$ for the constants $C_1$, $C_2$ and $\alpha$ appearing in Corollary~\ref{CrlBeurlIntDiscr2D} and Lemma~\ref{MakApplDiscr}.

We want to estimate $\Prob\left\{\rad(A_{M})>2R\right\}$. Take the first DBM branch $x_1, \ldots x_L$ (a collection of sites such that $x_{k+1}$ is adjacent to $x_k$ for all $k$) that starts at radius $R$ and reaches radius $2R$. Note that $L\geq R$. Suppose that these points were attached at times $k_1\ldots, k_L$.
We estimate below the number of possible paths taken by branches, the number of such branches, and probabilities that the corresponding paths are filled between times $N$ and $M$.
\begin{multline*}
\Prob\left\{\rad(A_{M})>2R\right\} \leq \\
\leq \sum_{L=R}^{M-N}
\begin{pmatrix}
\text{number of ways to choose} \\ 
L\text{ points out of }M-N
\end{pmatrix}
\begin{pmatrix}
\text{number of such} \\ 
\text{paths of length }L
\end{pmatrix}
\\
\times\Prob\begin{Bmatrix}
\text{a given path of length }L\text{ is filled} \\ 
\text{at the given moments}
\end{Bmatrix}.
\end{multline*}
Thus, using straightforward estimate of the number of such paths we write
\begin{equation*}
\Prob\left\{\rad(A_{M})>2R\right\}
\leq  \sum_{L=R}^{M} \binom{M}{L} \, 100 \, R \: 3^{L}\prod_{j=1}^{L}\frac{\hm^{\pr}(x_j, \overline{A_{k_j}})}{\sum_{y \in \bd A_{k_j}}\hm^{\pr}(y, \overline{A_{k_j}})}.
\end{equation*}
Here the factor $100 R$ accounts for the starting point of the branch, which is chosen on the circle of radius $R$.
By Corollary \ref{CrlBeurlIntDiscr2D},
\begin{equation*}
\prod_{j=1}^{L} \hm^{\pr}(x_j, \overline{A_{k_j}}) \leq \left(C_1\, L^{-1/2}\right)^{\pr L},
\end{equation*}
and by Lemma \ref{MakApplDiscr}, 
\begin{align*}
\prod_{j=1}^{L} \frac{1}{\sum_{y \in \bd A_{k_j}}\hm^{\pr}(y, \overline{A_{k_j}})} 
&\leq
\displaystyle\left(C_2^{-1}\frac{\big(\log (2R)\big)^{\alpha |1-\pr|}}{R^{1-\pr}}\right)^L\\
&\leq
\displaystyle\left(C_2^{-1}2^{\alpha|1-\pr|}\frac{(\log R)^{\alpha |1-\pr|}}{R^{1-\pr}}\right)^L,
\end{align*}
where we used that $R>10$.
Thus,
\begin{equation*}
\Prob\left\{\rad(A_{M})>2R\right\}
\leq
\sum_{L=R}^{M} \binom{M}{L} \, 100 \, R \: 3^{L} \left(C_0\frac{L^{-\pr/2}(\log R)^{\alpha |1-\pr|}}{R^{1-\pr}}\right)^L.
\end{equation*}
It is easy to see that the expression under summation sign reaches its maximum at $L = R$.
Indeed, the ratio of these expressions for $L$ and $L-1$ is
\begin{equation*}
\frac{M-L+1}{L}\cdot 3C_0 \cdot \left(\frac{(L-1)^{\pr/2}}{L^{\pr/2}}\right)^{L-1}L^{-\pr/2}
\cdot\frac{(\log R)^{\alpha |1-\pr|}}{R^{1-\pr}}
\end{equation*}
which is less than $1$ for $M$ given by \eqref{localMdefin} and the chosen $\widetilde{C}$:
\begin{multline*}
\frac{M-L+1}{L}\cdot 3C_0 \left(\frac{(L-1)^{\pr/2}}{L^{\pr/2}}\right)^{L-1}L^{-\pr/2}
\cdot\frac{(\log R)^{\alpha |1-\pr|}}{R^{1-\pr}}<\\
3C_0 M L^{-1-\pr/2} \frac{(\log R)^{\alpha |1-\pr|}}{R^{1-\pr}}<
3C_0 M \frac{(\log R)^{\alpha |1-\pr|}}{R^{(4-\pr)/2}}
< 3C_0\widetilde{C}.
\end{multline*}
Thus, 
\begin{equation*}
\Prob\left\{\rad(A_{M})>2R\right\}
\leq M \binom{M}{R} \, 100 \, R \: 3^{R}\left(C_0\frac{R^{-\pr/2}(\log R)^{\alpha |1-\pr|}}{R^{1-\pr}}\right)^R.
\end{equation*}
For the chosen $\widetilde{C}$, right-hand side is smaller than $2^{-R}$ for $R$ big enough. Indeed, using inequality $R! >R^R/e^{R}$, we obtain
\begin{equation*}
M \binom{M}{R} < \frac{e^{R} M^{R+1}}{R^R} < \frac{\left(e\,\widetilde{C}\right)^{R+1}}{R^R} \left(\frac{R^{(4-\pr)/2}}{(\log R)^{\alpha |1-\pr|}}\right)^{R+1},
\end{equation*}
and thus
\begin{equation*}
M \binom{M}{R} \, 100 \, R \: 3^{R}\left(C_0\frac{R^{-\pr/2}(\log R)^{\alpha |1-\pr|}}{R^{1-\pr}}\right)^R
\leq 
(100 e \,\widetilde{C} C_0)^{R+1} R^{3-\pr/2},
\end{equation*}
with right-hand side smaller than $2^{-R}$ for the chosen $\widetilde{C}$ and $R$ big enough.

Hence, for $R$ big enough 
\begin{equation*}
\Prob\left\{\rad(A_{M})>2R\right\}< 2^{-R}.
\end{equation*}
By Borel-Cantelli lemma almost surely $\rad(A_{M})<2R$ for $R$ big enough, which, together with \eqref{localMdefin}, finishes the proof. 
\end{proof}

\subsection{Proof for DBM in 3 dimensions}
\label{Section3Dim}
Our proof for $3$ dimensional case is similar to the $2$ dimensional argument.
\medskip

It is known that for all connected $A\subset \Z^3$ we have 
\begin{equation} \label{CapRad3DDiscr}
\frac{\rad(A)}{\log\rad(A)} \lesssim \cp(A) \lesssim \rad(A).
\end{equation} 
The upper bound follows from capacity monotonicity and capacity of the ball estimate (see, e.g. \citep[Proposition 6.5.2]{LawlerRW}), while the lower bound was proved in \citep[Proposition 3.1]{KestenBetter}.
\medskip

We start with a 3D version of the capacity increment estimate:
\begin{lmm} \label{LmmCapGrowDiscr3d}
Let $A\subset\Z^3$ be a compact set and $x\in \bd A.$ Let $B = \left\{x\right\}\cup A$. Then %
$ \cp^{-1} (\overline{A})-\cp^{-1} (\overline{B})  \eqv \hm^2(x, \overline{A})$.
\end{lmm}
\begin{rmrk}
This is equivalent to $\cp (\overline{B})-\cp (\overline{A}) \eqv \cp^2\left(\overline{A}\right)\hm^2(x, \overline{A})$.
\end{rmrk}
\begin{proof}
It is known (see, e.g., \citep[Proposition 6.5.2]{LawlerRW}) that $\exists\, \varkappa>0$ such that for any compact $K$ and $y\to \infty$,
\begin{equation} \label{localDiscr3dHitAsympt}
\Prob^y\left\{T_K<\infty\right\} = \frac{\varkappa \, \cp(K)}{|y|}\left(1+o(1)
\right). 
\end{equation}
Let $\overline{B} \backslash \overline{A} = \{x_1, \ldots x_k\}$, where $1\leq k\leq 6$.
Observe that for $y\to \infty$ from \eqref{localDiscr3dHitAsympt} we have
\begin{multline} \label{localDiscr3dCapDif}
\frac{\varkappa \left(\cp (\overline{B})-\cp (\overline{A})\right)}{|y|}\left(1+o(1)\right) = \Prob^y\left\{T_{\overline{B}}<\infty\right\}-\Prob^y\left\{T_{\overline{A}}<\infty\right\}
=  \\
= \sum_{j=1}^k \Prob^y\left\{S(T_{\overline{B}})=x_j \right\}\Prob^{x_j}\left\{T_{\overline{A}}=\infty\right\}.
\end{multline}
Note that from \citep[Proposition 6.5.1]{LawlerRW} and \citep[Proposition 6.5.4]{LawlerRW} we have
\begin{equation}\label{localDiscr3dHitPointAsympt}
\Prob^y\left\{S(T_{\overline{B}})=x_j \right\}
\eqv \frac{\Prob^{x_j}\left\{T_{\overline{B}}= \infty \right\}}{|y|}.
\end{equation}

If $l = \argmax_{j}\big\{\Prob^{x_j}\left\{T_{\overline{B}}= \infty \right\}\big\}$, then it is easy to see that
\begin{equation} \label{localDiscr3dEscAsympt1}
\displaystyle\Prob^{x_l}\left\{T_{\overline{B}}= \infty \right\} 
\eqv \Prob^{x_l}\left\{T_{\overline{A}}=\infty\right\}
\gtrsim \Prob^{x_j}\left\{T_{\overline{A}}=\infty\right\}, 
\quad \text{for all } 1\leq j\leq k.
\end{equation}
Indeed, it is obvious that $\Prob^{x_l}\left\{T_{\overline{B}}= \infty \right\} 
\leq \Prob^{x_l}\left\{T_{\overline{A}}=\infty\right\}$,
and by last-exit decomposition for any $1\leq j\leq k$,
\begin{align*}
\displaystyle\Prob^{x_j}\left\{T_{\overline{A}}=\infty\right\}
 &\displaystyle= \sum_{i=1}^k  \Prob^{x_i}\left\{T_{\overline{B}}= \infty \right\}\, G(x_i, x_j, \overline{A})\\
&\displaystyle\leq
\Prob^{x_l}\left\{T_{\overline{B}}= \infty \right\}\sum_{i=1}^k G(x_i, x_j, \{x\}),
\end{align*}
so
$ \Prob^{x_j}\left\{T_{\overline{A}}=\infty\right\}  \lesssim \Prob^{x_l}\left\{T_{\overline{B}}= \infty \right\}$, since $\sum_{i=1}^k G(x_i, x_j, \{x\})\lesssim 1$.

\medskip

Therefore, combining combining \eqref{localDiscr3dCapDif}, \eqref{localDiscr3dHitPointAsympt} and \eqref{localDiscr3dEscAsympt1} we get
\begin{equation*}
\displaystyle\cp (\overline{B})-\cp (\overline{A}) 
\eqv 
\Big(\Prob^{x_l}\left\{T_{\overline{A}}=\infty\right\} \Big)^2
.
\end{equation*}
From $\Prob^{x}\left\{T_{\overline{A}}=\infty\right\} = \frac{1}{6}\sum_j \Prob^{x_j}\left\{T_{\overline{A}}=\infty\right\} $ and \eqref{localDiscr3dEscAsympt1} we deduce that
\begin{equation*}\label{localDiscr3dEscAsympt2}
\displaystyle\Prob^{x_l}\left\{T_{\overline{A}}=\infty\right\} 
\eqv
\Prob^{x}\left\{T_{\overline{A}}=\infty\right\} 
=
\cp(\overline{A})\,\hm\!\left(x, \overline{A} \right).
\end{equation*}
The result now follows, since $\Prob^{x}\left\{T_{\overline{A}}=\infty\right\}$ is smaller than $1$.
\end{proof}

\begin{crl}\label{CrlBeurlIntDiscr3D}
Let $A_k\subset \Z^3$ for $k\in \N$. 
Assume that $A_{k+1}\backslash A_k = \{x_k\}$ and $\hm_k = \hm(x_k, \overline{A_k})$. 
Then for any $n$ and sequence $\left\{ {k_j}\right\}_{j=1}^m$ such that $\forall j\leq m: {k_j}\geq n$, we have
$$\left(\prod_{j=1}^{m}\hm_{k_j}\right)^{1/m} \leq C_1 \left(m\, \cp(A_n)\right)^{-1/2},$$
for some constant $C_1>0$.
\end{crl}

\begin{proof}
By Lemma \ref{LmmCapGrowDiscr3d},
\begin{equation*}
\sum_{j=1}^{m}\hm^2_{k_j} \lesssim \frac{1}{\cp(A_n)}.
\end{equation*}
Then we apply the inequality between geometric and arithmetic means%
\begin{equation*}
\left(\prod_{j=1}^{m}\hm_{k_j}\right)^{1/m} 
= \left(\prod_{j=1}^{m}\hm^2_{k_j}\right)^{1/(2m)}
\leq \left( \frac{\sum_{j=1}^{m}\hm^2_{k_j}}{m}\right)^{1/2}
\leq \left(\frac{1}{m\, \cp(A_n)}\right)^{1/2}.
\end{equation*}
\end{proof}

To proceed, we need the following analogue of Beurling's estimate for $\Z^3$.
\begin{prp}[\citep{LawlerInersections}, Theorem 2.5.2] \label{BeurlingDiscr3D}
For any compact $A\subset \Z^3$ and any $x\in \bd A$ we have $$\hm(x, \overline{A})\lesssim \frac{(\log \rad(A))^{1/2}}{\rad(A)}.$$
\end{prp}

\begin{rmrk}
The estimate in Proposition \ref{BeurlingDiscr3D} is sharp up to a multiplicative constant. The equality holds up to a multiplicative constant for the end points of segment $[0, r]\times\left\{0\right\}\times\left\{0\right\}$.
\end{rmrk}

\begin{rmrk}
We need Proposition \ref{BeurlingDiscr3D} only in order to estimate the normalizing term $\sum_{x\in \bd A} \hm^{\pr}(x, \overline{A})$ for $\pr<1$. For $\pr \geq 1 $ (and for DLA in particular) we will only need an integral analogue from Corollary \ref{CrlBeurlIntDiscr3D}.
\end{rmrk}

\begin{lmm} \label{lemmaDenomDiscr3D}
There exists $C_2 = C_2(\pr)>0$ such that for any connected set $A\subset \Z^3$ the following holds.  
\begin{enumerate}
\item For $\pr\geq 1$ we have $\sum_{x\in \bd A} \hm^{\pr}(x, \overline{A}) 
\geq  C_2\,|A|^{1-\pr}$.
\item For $\pr < 1$ we have  $\sum_{x\in \bd A} \hm^{\pr}(x, \overline{A}) 
\geq C_2\, r^{1-\pr}(\log r)^{(\pr-1)/2}$, where $r = \rad (A)$.
\end{enumerate}
\end{lmm}
\begin{proof}
For $\pr\geq 1$ we use the H{\"o}lder inequality
\begin{equation*}
|\bd A|^{\frac{\pr-1}{\pr}}\left(\sum_{x\in \bd A} \hm^{\pr}(x, \overline{A})\right)^{\frac{1}{\pr}} \geq \sum_{x\in \bd A} \hm(x, \overline{A}) = 1.
\end{equation*} 
So,
\begin{equation*}
\sum_{x\in \bd A} \hm^{\pr}(x, \overline{A}) \geq |\bd A|^{1-\pr} > C_2|A|^{1-\pr},
\end{equation*} 
since $|\bd A|\leq 6|A|$.
\medskip

For $\pr < 1$ we observe that
\begin{equation*}
\sum_{x\in \bd A} \hm^{\pr}(x, \overline{A}) 
\geq \max_{y\in \bd A}\left\{\hm(y, \overline{A})\right\}^{\pr-1}\sum_{x\in \bd A} \hm(x, \overline{A})
= \max_{y\in \bd A}\left\{\hm(y, \overline{A})\right\}^{\pr-1},
\end{equation*}
and use Beurling's estimate from Proposition \ref{BeurlingDiscr3D}. 
\end{proof}

Now we combine Corollary \ref{CrlBeurlIntDiscr3D} and Lemma \ref{lemmaDenomDiscr3D} to prove Theorem \ref{ThrDiscr3D}.

\begin{proof}[Proof of Theorem \ref{ThrDiscr3D}]

\medskip

For $\pr<1$, the proof is similar to the proof of Theorem \ref{ThrDiscr2D}. We again consider the first branch that reaches the circle of radius $2R$ and starts at the circle of radius $R$. We estimate the probability that such a branch is formed after $M$ particles have arrived in a similar way, as we did in the proof of Theorem \ref{ThrDiscr2D}. The only difference is that we use Corollary \ref{CrlBeurlIntDiscr3D} and Lemma \ref{lemmaDenomDiscr3D} instead of Corollary \ref{CrlBeurlIntDiscr2D} and Lemma \ref{MakApplDiscr}.

\medskip

Assume $\pr\geq1$. Let $R= \rad(A_N)$.
We will assume that $N$ is sufficiently large, so that $R>10$.
Set
\begin{equation*}
M :=\left\lfloor \widetilde{C}\frac{R^{\frac{1+\pr}{\pr}}}{(\log R)^{1/2}}\right \rfloor
\end{equation*}
for a small constant $\widetilde{C}$ given by
\begin{equation*}
\widetilde{C} =  (10^6 C_0)^{-1},
\end{equation*} 
where $C_0 = C_1^{\pr}C_2^{-1}$  for the constants $C_1$ and $C_2$ appearing in Corollary \ref{CrlBeurlIntDiscr3D} and Lemma \ref{lemmaDenomDiscr3D}.

We want to estimate $\Prob\left\{\rad(A_{M})>2R\right\}$.

If $\rad(A_M)>2R$ then $M\geq N$.
Analogously to the proof of Theorem \ref{ThrDiscr2D} we write 
\begin{equation}\label{eq3DProofMain}
\Prob\left\{\rad(A_{M})>2R \right\}
\leq \sum_{L=R}^{M} \binom{M}{L} 100 \, R^2\: 7^{L}\prod_{j=1}^{L}\frac{\hm^{\pr}(x_j, \overline{A}_{k_j})}{\sum_{y \in \bd A_{k_j}}\hm^{\pr}(y, \overline{A}_{k_j})},
\end{equation}
where the factor $100 R^2$ accounts for the starting point of the branch, which is chosen on the sphere of radius $R$.
We use Corollary \ref{CrlBeurlIntDiscr3D} and \eqref{CapRad3DDiscr} to get
\begin{equation*}
\prod_{j=1}^{L}\hm^{\pr}(x_j, \overline{A}_{k_j})
\leq
C_1^{\pr L} L^{-\pr L/2}\prod_{j=1}^{L} \cp(\overline{A}_{k_j})^{-\pr/2}
\leq
C_1^{\pr L} L^{-\pr L/2} \left(\frac{\log R}{R} \right)^{L\pr},
\end{equation*}
where we used that $(\log R)/R$ is decreasing for $R>10$.
From Lemma \ref{lemmaDenomDiscr3D} we get that
\begin{equation*}
\prod_{j=1}^{L}\frac{1}{\sum_{y \in \bd A_{k_j}}\hm^{\pr}(y, \overline{A}_{k_j})}
\leq
C_2^{-L} \prod_{j=1}^{L}|A_{k_j}|^{(\pr-1) }
\leq
C_2^{-L} M^{(\pr-1)L}.
\end{equation*}

Therefore, it follows from \eqref{eq3DProofMain} that
\begin{equation*}
\Prob\left\{\rad(A_{M})>2R \right\}
\leq
\sum_{L=R}^{M} \binom{M}{L} 100\, R^2\: 7^{L}\,
C_0^{L}\frac{\left(\log R\right)^{\pr L/2}M^{(\pr-1)L}}{R^{\pr L/2}L^{\pr L/2}}
\end{equation*}

Analogously to the proof of Theorem \ref{ThrDiscr2D} we obtain that for the chosen $\widetilde{C}$ and sufficiently large $R$,
\begin{equation*}
\Prob\left\{\rad(A_{M})>2R \right\}<2^{-R}.
\end{equation*}

Thus, by Borel-Cantelli lemma we get the desired result.
\end{proof}

\appendix

\section{Appendix}\label{sect_app}
Below we show that in 2 dimensions our definition of $G(x, A) = G(x, \infty, A)$ as the limit $\lim_{y\to \infty}G(x, y, A)$ agrees with the definition of Green's function from \citep[Proposition~6.4.7]{LawlerRW}.
\begin{prp}
In 2 dimensions for any finite set $A\subset \Z^2$ we have 
$$G(x, A) = a(x-y)-\E^x\left[a\left(S\!\left(\overline{T_A}\right)-y\right)\right],$$
for any $y\in A$.
\end{prp}
\begin{proof}
By \citep[Section~14, Theorem~3]{Spitzer_book} we have that
\begin{equation}\label{eq_app_1}
G(x, A)= \sum_{z\in A}a(x-z)w_A(z) - \cp(A).
\end{equation}
By the definition of the capacity \citep[Chapter 6.6]{LawlerRW} we get that for any $u\in A$,
$$\sum_{z\in A}a(u-z)w_A(z) = \cp(A).$$ 
Therefore, for any $u\in A$, $G(u, A)=0$. 
Moreover, combining \eqref{eq_app_1} and an asymptotic expansion of $a(x)$ \citep[Theorem 4.4.4]{LawlerRW} we get that
\begin{equation}\label{eq_app_2}
G(x, A) = a(x) +O(1), \qquad \text{as } x\to \infty.
\end{equation}
Thus, by \citep[Proposition~6.4.8]{LawlerRW}, there exists a constant $C\in \R$, such that
\begin{equation*}
G(x, A) = C \Big(a(x-y)-\E^x\left[a\left(S\!\left(\overline{T_A}\right)-y\right)\right]\Big),
\end{equation*} 
for any $y\in A$.
However, we also observe that
\begin{equation*}
a(x-y)-\E^x\left[a\left(S\!\left(\overline{T_A}\right)-y\right)\right] =  a(x) +O(1), \qquad \text{as } x\to \infty.
\end{equation*}
Therefore, combining this with \eqref{eq_app_2} we get that $C=1$, which finishes the proof.
\end{proof}

\section*{Declarations}

\textbf{Funding:}
Both authors are grateful to Swiss NSF and NCCR SwissMAP for financial support.

\bibliography{DBM}

\end{document}